\documentclass[12pt]{amsart}

\usepackage{amsrefs}
\usepackage{amssymb}
\usepackage{fancyhdr}
\usepackage{bbm}
\usepackage{xcolor}
\usepackage{hyperref}
\hypersetup{
  colorlinks=true,
  linkcolor=blue,
  citecolor=blue
}
\usepackage{mleftright}

\newtheorem{theorem}{Theorem}
\newtheorem*{theorem*}{Theorem}

\newtheorem{proposition}[theorem]{Proposition}
\newtheorem{claim}[theorem]{Claim}

\newtheorem{maintheorem}{Theorem}

\theoremstyle{definition}

\newtheorem*{definition*}{Definition}
\newtheorem*{lemma*}{Lemma}
\usepackage{graphicx}
\usepackage{pstricks, enumerate, pst-node, pst-text, pst-plot}

\numberwithin{equation}{section}
\numberwithin{theorem}{section}

\newcommand{\R}{\mathbb{R}}
\newcommand{\N}{\mathbb{N}}

\newcommand{\Z}{\mathbb{Z}}

\usepackage{xparse}
\DeclareDocumentCommand\Pr{ m g }{\ensuremath{
    {   \IfNoValueTF {#2}
      {\mathbb{P}\mleft[{#1}\mright]}
      {\mathbb{P}\mleft[{#1}\middle\vert{#2}\mright]}%
    }
}}
\DeclareDocumentCommand\E{ m g }{\ensuremath{
    {   \IfNoValueTF {#2}
      {\mathbb{E}\mleft[{#1}\mright]}
      {\mathbb{E}\mleft[{#1}\middle\vert{#2}\mright]}%
    }
}}

\def\dd{\mathrm{d}}

\def\aut{\mathrm{Aut}}
\def\hom{\mathrm{Homeo}}

\begin{document}

\title[]{Homomorphisms to $\mathbb{R}$ of automorphism groups of zero entropy shifts}

\author[]{Omer Tamuz}
\address{California Institute of Technology}


\thanks{Email: tamuz@caltech.edu. The author was supported by a grant from the Simons
Foundation (\#419427), a Sloan fellowship, a BSF award (\#2018397) and a National Science Foundation
CAREER award (DMS-1944153)}
\date{\today}

\begin{abstract}
We show that the automorphism group of every zero entropy infinite shift admits a ``drift'' homomorphism to $(\mathbb{R},+)$ that maps the shift map to 1. This homomorphism arises as the expectation, under an invariant measure, of a cocycle defined on a space of asymptotic pairs.
\end{abstract}

\maketitle
\section{Introduction}

Automorphism groups of low complexity shifts have attracted much attention in the past few years (see, e.g.,~\cite{cyr2015automorphism1,cyr2015automorphism,coven2015automorphisms,cyr2014automorphism, donoso2015automorphism, salo2014toeplitz, salo2014block}), and this paper builds on this work. We recall the basic definitions and state our main result.  

Let $A$ be a finite set called an \emph{alphabet}. The set $A^\Z$, endowed with the product topology, is called the \emph{full shift}. Let $\sigma \colon A^\Z \to A^Z$ denote the \emph{shift map}, given by $[\sigma(x)]_n = x_{n-1}$. A closed, $\sigma$-invariant subset of $A^\Z$ is called a \emph{shift}. An infinite shift can be either countable or uncountable.


Let $\Sigma \subseteq A^\Z$ be a shift.
A \emph{word} of length $n$ in $\Sigma$  is an element $w \in A^n$ such that $w = (x_1,\ldots,x_n)$ for some $x \in \Sigma$. The number of words of length $n$ in $\Sigma$ is denoted by $P_\Sigma(n)$. The entropy of $\Sigma$ is given by
\begin{align*}
    h(\Sigma) = \lim_n \frac{1}{n}\log P_\Sigma(n). 
\end{align*}

The automorphism group of a shift $\Sigma$, denoted $\aut(\Sigma)$, is the group of  homeomorphisms  of $\Sigma$ that commute with $\sigma$. Note that $\sigma$ (or, more precisely, its restriction to $\Sigma$) is an element of $\aut(\Sigma)$.

Our main result shows that when $\Sigma$ is zero entropy and infinite, then $\aut(\Sigma)$ is \emph{indicable}: it admits a non-trivial homomorphism to the additive group $(\R,+)$.
\begin{maintheorem}
\label{thm:main}
Let $\Sigma$ be a zero entropy infinite shift. Then there exists a group homomorphism $\Phi \colon \aut(\Sigma) \to \R$ such that $\Phi(\sigma) = 1$.
\end{maintheorem}

We construct $\Phi$ by defining a action of $\aut(\Sigma)$ by homeomorphisms on a space $CA(\Sigma)$ of \emph{asymptotic pairs}. We show that this space admits a bounded ``drift'' cocycle $c \colon \aut(\Sigma) \times CA(\Sigma) \to \Z$ that satisfies $c(\sigma,\cdot) = 1$. Furthermore, using a technique introduced in \cite{frisch2021characteristic}, we show that $CA(\Sigma)$ also admits an $\aut(\Sigma)$-invariant probability measure $\nu$. The drift homomorphism $\Phi$ is defined as the expectation of $c$ with respect to $\nu$: $\Phi(\varphi) = \int c(\varphi,\cdot)\,\dd\nu$. 

This homomorphism has a similar flavor to those that stem from Krieger's dimension representation \cite{krieger1980dimension, boyle1988automorphism}. The construction of a homomorphism through the integration of a cocycle  with respect to an invariant measure is a technique that has yielded other interesting results in the past (see, e.g., Karlsson and Ledrappier~\cite{karlsson2006laws}).

The remainder of this paper contains definitions and a proof of Theorem~\ref{thm:main}. In \S\ref{sec:example} we provide some examples and further notes.

\medskip

\subsubsection*{Acknowledgment.} The author thanks Joshua Frisch and Ville Salo for valuable comments on an earlier version.

\section{The space of calibrated asymptotic pairs}
As is well known (see, e.g., \cite[Chapter 2]{auslander1988minimal}) every infinite shift $\Sigma$  admits at least one \emph{asymptotic pair}: $x,y \in \Sigma$ such that $x_M \neq y_M$ for some $M \in \Z$, and $x_n = y_n$ for all $n < M$. Accordingly, given an asymptotic pair, we denote by
\begin{align*}
    M(x,y) = \min \{m \in \Z \:\,: x_m \neq y_m\}
\end{align*}
the first coordinate in which $x$ and $y$ differ. Note that 
\begin{align}
    \label{eq:M}
    M(\sigma^k x, \sigma^k y) = M(x,y)+k.
\end{align}

Asymptotic pairs have been used to study automorphism groups of shifts: in  \cite{donoso2015automorphism} it is shown that $\aut(\Sigma)$ is virtually $\Z$ if $\Sigma$ is transitive and $\liminf_n P_\Sigma(n)/n$ is finite.

We say that an asymptotic pair is \emph{calibrated} if $M(x,y)=0$. If $(x,y)$ is an asymptotic pair then $(\sigma^m(x),\sigma^m(y))$ is an asymptotic pair, and it is calibrated if and only if $m = -M(x,y)$. We denote by
\begin{align*}
    C(x, y) = (\sigma^{-M(x,y)} x, \sigma^{-M(x,y)}y)
\end{align*}
the calibrated asymptotic pair that is attained from $(x,y)$ by shifting both of them so that they first differ at $0$. We denote by $CA(\Sigma)$ the set of calibrated asymptotic pairs in $\Sigma$. This definition is closely related to the \emph{asymptotic components} of \cite{donoso2016automorphism} and the \emph{asymptotic composants} of \cite{barge2001complete}. It is straightforward to see that $CA(\Sigma)$ is a closed subset of $\Sigma^2$, and is therefore compact. Note also that $C(x,y) \in CA(\Sigma)$ for every asymptotic pair $(x,y)$.

\section{An $\aut(\Sigma)$ action on the calibrated asymptotic pairs}
Let $\Sigma$ be an infinite shift. Then $\Sigma$ admits an asymptotic pair, and so $CA(\Sigma)$ is non-empty. We construct an $\aut(\Sigma)$ action on $CA(\Sigma)$. Given an automorphism $\varphi$ of $\Sigma$, define $\hat \varphi \colon CA(\Sigma) \to CA(\Sigma)$ by
\begin{align}
  \label{eq:ca-action}
    \hat \varphi(x,y) = C(\varphi x ,\varphi y).
\end{align}
That is, given a calibrated asymptotic pair $(x,y)$, $\hat\varphi$ applies $\varphi$ to both $x$ and $y$, and then shifts the resulting asymptotic pair so that it is again calibrated. The next few claims show that this is a well defined action by homeomorphisms. While $\hat\varphi$ is easily seen to be measurable, its continuity is less apparent. 

By the Curtis-Lyndon-Hedlund Theorem~\cite{hedlund1969endomorphisms}, for every $\varphi \in \aut(\Sigma)$ there is a \emph{memory} $k \in \N$ and a \emph{block map} $B_\varphi \colon A^{\{-k,\ldots,k\}} \to A$ such that $[\varphi x]_m = B_\varphi(x_{m-k},\ldots,x_{m+k})$. Importantly, $[\varphi x]_m$ is determined by $(x_{m-k},\ldots,x_{m+k})$. The following claim, which shows that $\varphi$ is well defined, is a direct consequence.
\begin{claim}
\label{clm:aut-inv}
If $(x,y)$ is an asymptotic pair in $\Sigma$ then so is $(\varphi x,\varphi y)$, for any $\varphi \in \aut(\Sigma)$.
\end{claim}
\begin{proof}
  Let $k$ be a memory of $\varphi$. Then $[\varphi x]_m = [\varphi y]_m$ for all $m < M(x,y)-k$, since $(x_{m-k},\ldots,x_{m+k}) = (y_{m-k},\ldots,y_{m+k})$ for such $m$. And $\varphi x \neq \varphi y$, since $\varphi$ is a bijection, and since $x \neq y$.
\end{proof}

The next claim offers a bound on the difference between $M(x,y)$ and $M(\varphi x,\varphi y)$. This will be the key component in the proof that $\hat\varphi$ is continuous.
\begin{claim}
\label{clm:m-bdd}
For every $\varphi \in \aut(\Sigma)$ there is a $B \in \N$ such that $|M(x,y)-M(\varphi x,\varphi y)| \leq B$ for every asymptotic pair $(x,y)$ in $\Sigma$.
\end{claim}
\begin{proof}
Let $k$ be a memory of $\varphi$. Then $[\varphi x]_m = [\varphi y]_m$ for all $m < M(x,y)-k$, as noted in the proof of Claim~\ref{clm:aut-inv} above. Hence 
\begin{align*}
M(\varphi x,\varphi y) \geq M(x,y)-k,    
\end{align*}
which provides one side of the desired inequality.

Now, let $k'$ be a memory of $\varphi^{-1}$. Then by the same argument applied to the pair $(\varphi x,\varphi y)$ and the automorphism $\varphi^{-1}$ we have that
\begin{align*}
 M(\varphi^{-1} \varphi x,\varphi^{-1} \varphi y) \geq M(\varphi x,\varphi y)-k',   
\end{align*}
which provides the other side of the inequality. Thus the claim holds for $B = \max\{k,k'\}$.
\end{proof}

\begin{proposition}
\label{prop:continuous}
Each map $\hat\varphi \colon CA(\Sigma) \to CA(\Sigma)$ is continuous.
\end{proposition}
\begin{proof}
Fix $\varphi \in \aut(\Sigma)$. By definition, $M(x,y)=0$ for $(x,y) \in CA(\Sigma)$. Thus, by Claim~\ref{clm:m-bdd}, there is some $B$ such that $|M(\varphi x,\varphi y)| \leq B$ for all $(x,y) \in CA(\Sigma)$. Thus the map $M_\varphi \colon CA(\Sigma) \to \Z$ given by  $M_\varphi(x,y) = M(\varphi x,\varphi y)$ takes values in $\{-B,\ldots,B\}$. Furthermore
\begin{align*}
    M_\varphi(x,y) = \min\{m \in \{-B,\ldots,B\}\,:\, [\varphi x]_m \neq [\varphi y]_m\},
\end{align*}
and so $M_\varphi$ is continuous, since $\varphi$ is continuous. Since
\begin{align*}
    \hat \varphi(x,y) = (\sigma^{-M_\varphi(x,y)}\varphi x,\sigma^{-M_\varphi(x,y)}\varphi y),
\end{align*}
and again using that $\varphi$ is continuous, it follows that $\hat\varphi$ is also continuous.
\end{proof}
The proof above in fact shows a stronger claim, which will be important later:
\begin{claim}
\label{clm:phi-extra-cont}
  For each $\varphi \in \aut(\Sigma)$ there is a $b \in \N$ such that the $m$th coordinates of $\hat\varphi(x,y)$ depend only on  $(x_{-m-b},\ldots,x_{m+b})$ and $(y_{-m-b},\ldots,y_{m+b})$.
\end{claim}

Finally, the next claim completes the proof that we have defined an action of $\aut(\Sigma)$ on $CA(\Sigma)$.
\begin{claim}
\label{clm:homomorph}
The map $\varphi \mapsto \hat\varphi$ is a group homomorphism from $\aut(\Sigma)$ to $\hom(CA(\Sigma))$.
\end{claim}
\begin{proof}
By Proposition~\ref{prop:continuous}, each $\hat\varphi$ is a homeomorphism of the compact space $CA(\Sigma)$. It thus suffices to prove that $\hat \rho \hat \varphi = \widehat{\rho\varphi}$ for all $\rho,\varphi \in \aut(\Sigma)$.

By definition
\begin{align*}
    \hat \rho\hat \varphi(x,y) 
    &= \hat\rho(\sigma^{m} \varphi x, \sigma^{m}\varphi y)
\end{align*}
where $m = -M(\varphi x,\varphi y)$. Applying the definition again we get
\begin{align*}
    \hat \rho\hat \varphi(x,y) 
    &= (\sigma^n\rho\sigma^{m} \varphi x, \sigma^n\rho\sigma^{m}\varphi y),
\end{align*}
where $n = -M(\rho \sigma^{m} \varphi x, \rho \sigma^{m} \varphi y)$. Since $\rho$ commutes with $\sigma$, $n = -M(\sigma^{m}\rho \varphi x, \sigma^{m}\rho \varphi y)$, and by \eqref{eq:M}, $n = -M(\rho \varphi x, \rho \varphi y)-m$. Thus, and by again using the fact that $\sigma$ and $\rho$ commute, 
\begin{align*}
    \hat \rho\hat \varphi(x,y) 
    &= (\sigma^{n+m}\rho \varphi x, \sigma^{n+m}\rho\varphi y)\\
    &= (\sigma^{-M(\rho \varphi x, \rho \varphi y)}\rho \varphi x, \sigma^{-M(\rho \varphi x, \rho \varphi y)}\rho \varphi y)\\
    &= \widehat{\rho\varphi}(x,y).
\end{align*}
\end{proof}

\section{The drift cocycle and drift homomorphisms}

Define the \emph{drift cocycle} $c \colon \aut(\Sigma) \times CA(\Sigma) \to \Z$ by
\begin{align*}
    c(\varphi, (x,y)) = M(\varphi x,\varphi y).
\end{align*}
In a sense, $c(\varphi,(x,y))$ captures the amount by which $\varphi$ shifts the asymptotic pair $(x,y)$. In particular, by \eqref{eq:M}, $c(\sigma,(x,y))=1$ for all $(x,y) \in CA(\Sigma)$.
\begin{claim}
\label{clm:c-cont-cocycle}
The drift cocycle $c$ is continuous and satisfies the cocycle relation
\begin{align*}
    c(\varphi \rho, (x,y)) = c(\varphi, \hat\rho(x,y)) + c(\hat\rho, (x,y)).
\end{align*}
\end{claim}
\begin{proof}
In the proof of Proposition~\ref{prop:continuous} we defined $M_\varphi(x,y) = M(\varphi x,\varphi y) = c(\varphi,(x,y))$ and showed that it is continuous, and thus $c$ is continuous. It remains to be shown that it satisfies the cocycle relation. 

By definition,
\begin{align*}
    c(\varphi, \rho(x,y)) 
    &= c(\varphi, C(\rho x ,\rho y ))\\
    &= c(\varphi, (\sigma^{-M(\rho x ,\rho y )}\rho x , \sigma^{-M(\rho x ,\rho y )}\rho y ))\\
    &= M(\varphi \sigma^{-M(\rho x ,\rho y )}\rho x , \varphi\sigma^{-M(\rho x ,\rho y )}\rho y ).
\end{align*}
Since $\varphi$ commutes with $\sigma$, we have that 
\begin{align*}
     c(\varphi, \rho(x,y)) = M( \sigma^{-M(\rho x ,\rho y )}\varphi\rho x , \sigma^{-M(\rho x ,\rho y )}\varphi\rho y ).
\end{align*}
Applying \eqref{eq:ca-action} yields
\begin{align*}
    c(\varphi, \rho(x,y)) = M(\varphi\rho x,\varphi\rho y) - M(\rho x,\rho y),
\end{align*}
which is equal to $c(\varphi\rho,(x,y))-x(\rho,(x,y))$.
\end{proof}

Since $M(x,y)=0$ for all $(x,y) \in CA(\Sigma)$, by Claim~\ref{clm:m-bdd} $c$ is a bounded cocycle:
\begin{claim}
\label{clm:bdd-cocycle}
  For each $\varphi \in \aut(\Sigma)$ there exists a $B \in \N$ such that $|c(\varphi,(x,y))| \leq B$ for all $(x,y) \in CA(\Sigma)$.
\end{claim}

Suppose $\nu$ is a Borel probability measure on $CA(\Sigma)$ that is $\aut(\Sigma)$-invariant, i.e., $\nu(\hat\varphi(A)) = \nu(A)$ for every $\varphi \in \aut(\Sigma)$ and Borel $A \subset CA(\Sigma)$. Let $\Phi_\nu \colon \aut(\Sigma) \to \R$ be given by
\begin{align*}
    \Phi_\nu(\varphi) = \int c(\varphi,(x,y))\,\dd\nu(x,y).
\end{align*}
We call $\Phi_\nu$ a \emph{drift} homomorphism.
\begin{claim}
  $\Phi_\nu$ is a well defined homomorphism $\aut(\Sigma) \to \R$, with $\Phi_\nu(\sigma)=1$.
\end{claim}
\begin{proof}
  By Claim~\ref{clm:c-cont-cocycle} $c(\varphi,\cdot)$ is continuous and thus measurable. By Claim~\ref{clm:bdd-cocycle} it is bounded. Hence it is integrable, and $\Phi_\nu(\varphi)$ is well defined. To show that it is a homomorphism we first apply the cocycle relation, and then the invariance of $\nu$
  \begin{align*}
      \Phi_\nu(\rho \varphi) 
      &= \int c( \varphi\rho,(x,y))\,\dd\nu(x,y)\\ 
      &= \int c(\varphi,\hat\rho(x,y))\,\dd\nu(x,y)+\int c(\rho,(x,y))\,\dd\nu(x,y)\\
      &= \int c(\varphi,(x,y))\,\dd\nu(x,y)+\int c(\rho,(x,y))\,\dd\nu(x,y)\\
      &= \Phi_\nu(\varphi) + \Phi_\nu(\rho).
  \end{align*}
  Finally, $\Phi_\nu(\sigma)=1$, since $c(\sigma,(x,y))=1$ for all $(x,y)$.
\end{proof}

In light of this claim, we prove our main theorem by showing that if $\Sigma$ is zero entropy shift with an asymptotic point then $CA(\Sigma)$ admits an $\aut(\Sigma)$-invariant measure. This is what we do in the next section.

\section{$\aut(\Sigma)$-invariant random calibrated asymptotic pairs}

In this section we construct, for each zero entropy shift $\Sigma$, a Borel probability measure on $CA(\Sigma)$ that is $\aut(\Sigma)$-invariant. This construction is nearly identical to that of the main result in \cite{frisch2021characteristic}; we provide the details for completeness.

Let $W_n \subseteq A^{2n+1} \times A^{2n+1}$ denote the set of \emph{word-pairs} that appear in the centered window of radius $n$ in $CA(\Sigma)$. That is, $W_n$ is the set of  $(w_1,w_2) \in A^{2n+1} \times A^{2n+1}$ such that $w_1=(x_{-n},\ldots,x_n)$ and $w_2 = (y_{-n},\ldots,y_n)$ for some $(x,y)\in CA(\Sigma)$. We say that $(x,y)$ \emph{projects} to $(w_1,w_2) \in W_n$ if $(w_1,w_2)$ appears in $(x,y)$, and denote $\pi_n(x,y) = (w_1,w_2)$, $\pi_n \colon CA(\Sigma) \to W_n$. As each $(w_1,w_2) \in W_n$ appears in some pair $(x,y) \in CA(\Sigma)$, we can find a set $\bar W_n \subseteq CA(\Sigma)$ of the same size as $W_n$, and where each $(w_1,w_2) \in W_n$ appears in exactly one $(x,y) \in \bar W_n$. I.e., $\pi_n$ restricted to $\bar W_n$ is a bijection.

Since $|W_n| \leq (P_\Sigma(2n+1))^2$, and since $\Sigma$ has zero entropy, $|W_n|$ grows sub-exponentially, and so there is a sequence $(n_m)_m$ such that
\begin{align*}
    \frac{|W_{n_m+m}|}{|W_{n_m}|} \leq 1+o_m(1).
\end{align*}
That is, along the sequence $n_m$, the number of word-pairs in a window of width $n_m+m$ is only a small fraction more than in a window of length $n_m$. It follows that, along this sequence,  at least a $1-o(1)$ fraction of the word-pairs in $W_{n_m}$ have a unique extension to a word-pair in $W_{n_m+m}$. Denote such a sequence of sets by $U_m$. 

Let $\nu_m$ be the uniform measure on $\bar W_{n_m}$, which, we remind the reader, is a subset of $CA(\Sigma)$ for which the projection to the words $W_{n_m}$ is a bijection. Since $CA(\Sigma)$ is compact, the sequence $\nu_m$ has a subsequential limit $\nu$.
\begin{proposition}
\label{prop:invariant}
The measure $\nu$ on $CA(\Sigma)$ is $\aut(\Sigma)$-invariant.
\end{proposition}
\begin{proof}

Since $\nu$ is defined on the Borel sigma-algebra, to show that it is invariant it suffices to show that $\nu(\hat\varphi^{-1} \bar E) = \nu(\bar E)$ for every $\varphi \in \aut(\Sigma)$ and every clopen $\bar E \subseteq CA(\Sigma)$.

Since $\bar E$ is clopen, for all $m$ large enough there is a set $E_m \subset W_{n_m}$ such that $\bar E$ is the set of all $(x,y) \in CA(\Sigma)$ that project to some $(w_1,w_2) \in E_m$. That is, $\bar E = \pi_{n_m}^{-1}(E_m)$. Hence $|\bar E \cap \bar W_{n_m}| = |E_m \cap W_{n_m}| = |E_m|$, where the last equality holds since $E_m \subseteq W_{n_m}$. It follows that 
\begin{align*}
    \nu_m(\bar E) = \frac{|\bar E \cap \bar W_{n_m}|}{|\bar W_{n_m}|}
    =\frac{|E_m \cap W_{n_m}|}{|W_{n_m}|}
    =\frac{|E_m|}{|W_{n_m}|},
\end{align*}
where the first equality is the definition of $\nu_m$. Likewise,
\begin{align*}
    \nu_m(\hat\varphi^{-1}\bar E) 
    = \frac{|(\hat\varphi^{-1}\bar E) \cap \bar W_{n_m}|}{|\bar W_{n_m}|}
    = \frac{|\bar E \cap \hat\varphi\bar W_{n_m}|}{|\bar W_{n_m}|}.
\end{align*}
Now, $\bar E \cap \hat\varphi\bar W_{n_m}$ is the set of elements of $\hat\varphi\bar W_{n_m}$ that are in $\bar E$. Since $\bar E = \pi_{n_m}^{-1}(E_m)$, The size of this set is equal to the number of elements $(x,y) \in \hat\varphi\bar W_{n_m}$ such that $\pi_{n_m}(x,y)$ is in $E_m$:
\begin{align*}
    \nu_m(\hat\varphi^{-1}\bar E) 
    = \frac{|\{(x,y) \in \hat\varphi\bar W_{n_m}\,:\, \pi_{n_m}(x,y) \in E_m\}|}{|\bar W_{n_m}|}.
\end{align*}
By Claim~\ref{clm:phi-extra-cont} there is some $b$ such that for all $n$, coordinates $\{-n,\ldots,n\}$ of both $\hat\varphi(x,y)$ and $\hat\varphi^{-1}(x,y)$ are determined by coordinates $(-n-b,\ldots,n+b)$ of $(x,y)$. For $m > b$, the set of word-pairs $U_m \subseteq W_{n_m}$ that have a unique extension to $W_{n_m+m}$ is at least a $1-o(1)$ fraction of the word-pairs in $W_{n_m}$. Let $\bar U_m \subset \bar W_{n_m}$ be the subset of  $\bar W_{n_m}$ that projects to $U_m$. Since $\bar U_m$ is almost all of $\bar W_{n_m}$, we can substitute it for $\bar W_{n_m}$ and only incur a vanishing error:
\begin{align*}
    \nu_m(\hat\varphi^{-1}\bar E) 
    = \frac{|\{(x,y) \in \hat\varphi\bar U_m\,:\, \pi_{n_m}(x,y) \in E_m\}|}{|\bar W_{n_m}|} + o(1).
\end{align*}

The key observation is that because of the unique extension property, for $(w_1,w_2) \in U_m$ there is a $(w_1',w_2')$ such that  $(x,y)$ projects to $(w_1,w_2)$ if and only if $\hat\varphi(x,y)$ projects to $(w_1',w_2')$. This holds because these projections are determined by $(x_{-n_m-b},\ldots,x_{n_m+b})$, which by the unique extension property is determined by $(x_{-n_m},\ldots,x_{n_m})$. Hence the restriction of $\pi_{n_m}$ to $\hat\varphi\bar U_m$ is injective.
Hence
\begin{align*}
    \nu_m(\hat\varphi^{-1}\bar E) 
    = \frac{|E_m \cap \pi_{n_m}(\hat\varphi\bar U_m)|}{|W_{n_m}|} + o(1).
\end{align*}
Since $\bar U_m$ includes $1-o(1)$ of the elements of $\bar W_{m_n}$, and since $\pi_{n_m} \circ \hat\varphi$ is injective on it, we can replace $\bar U_m$ by $W_{n_m}$ while again only incurring an additional vanishing error:
\begin{align*}
    \nu_m(\hat\varphi^{-1}\bar E) 
    = \frac{|E_m \cap W_{n_m}|}{|W_{n_m}|} + o(1).
\end{align*}
Finally, $E_m$ is a subset of $W_{n_m}$, and so 
\begin{align*}
    \nu_m(\hat\varphi^{-1}\bar E) 
    = \frac{|E_m|}{|W_{n_m}|} + o(1) = \nu_m(\bar E)+o(1).
\end{align*}
By taking the limit $m \to \infty$ it follows that $\nu(\hat\varphi^{-1}\bar E)=  \nu(\bar E)$, and so $\nu$ is $\hat \varphi$-invariant.
\end{proof}

\section{Proof of Theorem~\ref{thm:main}}
The proof of Theorem~\ref{thm:main} is an immediate consequence of Claim~\ref{clm:homomorph} and Proposition~\ref{prop:invariant}: By Proposition~\ref{prop:invariant}, there is an $\aut(\Sigma)$-invariant probability measure on $CA(\Sigma)$, provided that $CA(\Sigma)$ is non-empty, which holds if $\Sigma$ no periodic points, and provided that $\Sigma$ is zero entropy. By Claim~\ref{clm:homomorph}, there exists an associated drift homomorphism $\Phi_\nu$ satisfies $\Phi_\nu(\sigma)=1$.

\section{Examples and further notes}
\label{sec:example}
The author would like to thank Joshua Frisch and Ville Salo for drawing his attention to the following examples.

As an example of an asymptotic pair, let $S \subset \{0,1\}^\Z$ be the \emph{sunny side up} shift: $S = \{x \in \{0,1\}^\Z \,:\, \sum_n x_n \leq 1\}$. That is, $S$ is the $\sigma$-orbit closure of $\bar x$, where $\bar x_0=1$ and $\bar x_n=0$ for $n \neq 0$. Let $\bar z_n=0$ for all $n$. Then $(\bar x, \bar z)$ is a calibrated asymptotic pair in $S$. In fact, there is only one more: $(\bar z,\bar x)$.

The \emph{topological full group} of $\Sigma$ is the set of homeomorphisms $\phi$ of $\Sigma$ for which there exists a continuous \emph{orbit cocycle} $N_\phi \colon \Sigma \to \Z$ for which $\phi(x) = \sigma^{N_\phi(x)}(x)$ (see, e.g., Katzlinger's survey~\cite{katzlinger2019topological}). Topological full groups admit drift homomorphisms that closely resemble our construction: they arise as expectations  of the orbit cocycle $N_\phi$ with respect to a shift-invariant measure on $\Sigma$.

Indeed, the relation to this paper can be made more explicit. Given any shift $\Sigma$, let $\Sigma' = \Sigma \times S$. The topological full group embeds as a subgroup of $\aut(\Sigma')$: given an element of the full group $\phi$, define $\bar \phi \in \aut(\Sigma')$ as follows. An element of $\Sigma'$ is either of the form $(y,\sigma^m\bar x)$ for $m \in \Z$ or of the form $(y, \bar z)$. In the latter case let $\bar \phi(y,\bar z) = (y,\bar z)$. In the former case let
\begin{align*}
    \bar \phi(y,\sigma^m \bar x) = (y, \sigma^{m+N_\phi(\sigma^{-m} y)} \bar x).
\end{align*}
Consider calibrated asymptotic pairs of the form $((y,\bar x), (y,\bar z)) \in CA(\Sigma')$. On such pairs, the drift cocycle $c$ equals to $N_\phi$. Hence, if we choose $y$ at random according to a shift-invariant probability measure on $\Sigma$, the expectation of the drift cocycle $c(\phi, ((y,\bar x), (y,\bar z)))$ will equal the expectation of the orbit cocycle $N_\phi(y)$.

\medskip 

Yet another similar construction of a drift homomorphism (which we will not explain in detail) is the ``average movement'' of Turing machines of shifts, as defined by Barbieri, Kari and Salo~\cite{barbieri2016group}; this is a generalization of the drift homomorphisms of full groups.

\bibliography{refs}

\begin{bibdiv}
\begin{biblist}

\bib{auslander1988minimal}{book}{
      author={Auslander, Joseph},
       title={Minimal flows and their extensions},
   publisher={Elsevier},
        date={1988},
}

\bib{barbieri2016group}{inproceedings}{
      author={Barbieri, Sebasti{\'a}n},
      author={Kari, Jarkko},
      author={Salo, Ville},
       title={The group of reversible {T}uring machines},
organization={Springer},
        date={2016},
   booktitle={International workshop on cellular automata and discrete complex
  systems},
       pages={49\ndash 62},
}

\bib{barge2001complete}{article}{
      author={Barge, Marcy},
      author={Diamond, Beverly},
       title={A complete invariant for the topology of one-dimensional
  substitution tiling spaces},
        date={2001},
     journal={Ergodic Theory and Dynamical Systems},
      volume={21},
      number={5},
       pages={1333\ndash 1358},
}

\bib{boyle1988automorphism}{article}{
      author={Boyle, Mike},
      author={Lind, Douglas},
      author={Rudolph, Daniel},
       title={The automorphism group of a shift of finite type},
        date={1988},
     journal={Transactions of the American Mathematical Society},
      volume={306},
      number={1},
       pages={71\ndash 114},
}

\bib{coven2015automorphisms}{article}{
      author={Coven, Ethan~M.},
      author={Quas, Anthony},
      author={Yassawi, Reem},
       title={Computing automorphism groups of shifts using atypical
  equivalence classes},
        date={2016},
     journal={Discrete Analysis},
       pages={Paper No. 3, 28},
         url={https://doi.org/10.19086/da.611},
}

\bib{cyr2015automorphism1}{article}{
      author={Cyr, Van},
      author={Kra, Bryna},
       title={The automorphism group of a shift of linear growth: beyond
  transitivity},
        date={2015},
     journal={Forum of Mathematics. Sigma},
      volume={3},
       pages={Paper No. e5, 27},
         url={https://doi.org/10.1017/fms.2015.3},
}

\bib{cyr2015automorphism}{article}{
      author={Cyr, Van},
      author={Kra, Bryna},
       title={The automorphism group of a minimal shift of stretched
  exponential growth},
        date={2016},
        ISSN={1930-5311},
     journal={Journal of Modern Dynamics},
      volume={10},
       pages={483\ndash 495},
         url={https://doi.org/10.3934/jmd.2016.10.483},
}

\bib{cyr2014automorphism}{article}{
      author={Cyr, Van},
      author={Kra, Bryna},
       title={The automorphism group of a shift of subquadratic growth},
        date={2016},
     journal={Proceedings of the American Mathematical Society},
      volume={144},
      number={2},
       pages={613\ndash 621},
}

\bib{donoso2015automorphism}{article}{
      author={Donoso, Sebasti\'{a}n},
      author={Durand, Fabien},
      author={Maass, Alejandro},
      author={Petite, Samuel},
       title={On automorphism groups of low complexity subshifts},
        date={2016},
        ISSN={0143-3857},
     journal={Ergodic Theory and Dynamical Systems},
      volume={36},
      number={1},
       pages={64\ndash 95},
         url={https://doi.org/10.1017/etds.2015.70},
}

\bib{donoso2016automorphism}{article}{
      author={Donoso, Sebasti{\'a}n},
      author={Durand, Fabien},
      author={Maass, Alejandro},
      author={Petite, Samuel},
       title={On automorphism groups of low complexity subshifts},
        date={2016},
     journal={Ergodic Theory and Dynamical Systems},
      volume={36},
      number={1},
       pages={64\ndash 95},
}

\bib{frisch2021characteristic}{article}{
      author={Frisch, Joshua},
      author={Tamuz, Omer},
       title={Characteristic measures of symbolic dynamical systems},
        date={2021},
     journal={Ergodic Theory and Dynamical Systems},
       pages={1\ndash 7},
}

\bib{hedlund1969endomorphisms}{article}{
      author={Hedlund, Gustav~A.},
       title={Endomorphisms and automorphisms of the shift dynamical system},
        date={1969},
        ISSN={0025-5661},
     journal={Mathematical systems theory},
      volume={3},
      number={4},
       pages={320\ndash 375},
}

\bib{karlsson2006laws}{article}{
      author={Karlsson, Anders},
      author={Ledrappier, Fran{\c{c}}ois},
       title={On laws of large numbers for random walks},
        date={2006},
     journal={The Annals of Probability},
      volume={34},
      number={5},
       pages={1693\ndash 1706},
}

\bib{katzlinger2019topological}{article}{
      author={Katzlinger, Leonhard},
       title={Topological full groups},
        date={2019},
     journal={arXiv preprint arXiv:1907.07424},
}

\bib{krieger1980dimension}{article}{
      author={Krieger, Wolfgang},
       title={On dimension functions and topological markov chains},
        date={1980},
     journal={Inventiones mathematicae},
      volume={56},
      number={3},
       pages={239\ndash 250},
}

\bib{salo2014toeplitz}{article}{
      author={Salo, Ville},
       title={Toeplitz subshift whose automorphism group is not finitely
  generated},
        date={2017},
        ISSN={0010-1354},
     journal={Colloquium Mathematicum},
      volume={146},
      number={1},
       pages={53\ndash 76},
         url={https://doi.org/10.4064/cm6463-2-2016},
}

\bib{salo2014block}{article}{
      author={Salo, Ville},
      author={T{\"o}rm{\"a}, Ilkka},
       title={Block maps between primitive uniform and pisot substitutions},
        date={2015},
     journal={Ergodic Theory and Dynamical Systems},
      volume={35},
      number={7},
       pages={2292\ndash 2310},
}

\end{biblist}
\end{bibdiv}

\end{document}